\documentclass[12pt,a4paper]{amsart}
\PassOptionsToPackage{showonlyrefs}{mathtools}
\usepackage{preamble}
\usepackage{mathtools}
\usepackage{amsmath}
\usepackage{amssymb}
\usepackage{mathrsfs}
\usepackage{indentfirst}
\allowdisplaybreaks

\declaretheorem[sibling=thm]{corollary}
\declaretheorem[name=Main Theorem,sibling=thm]{maintheorem}



\makeatletter
\@ifpackageloaded{hyperref}{}{\usepackage[colorlinks=true,linkcolor=blue,citecolor=blue,urlcolor=blue]{hyperref}}
\@ifpackageloaded{cleveref}{}{\usepackage[nameinlink,noabbrev,capitalise]{cleveref}}
\makeatother
\crefname{maintheorem}{Main Theorem}{Main Theorems}
\Crefname{maintheorem}{Main Theorem}{Main Theorems}
\crefname{corollary}{Corollary}{Corollaries}
\Crefname{corollary}{Corollary}{Corollaries}
\crefname{section}{Section}{Sections}
\Crefname{section}{Section}{Sections}
\crefname{enumi}{item}{items}
\Crefname{enumi}{Item}{Items}

\title[SRB Measures for $C^{1+\mathrm{Dini}}$ Diffeomorphisms]{SRB Measures for $C^{1+\mathrm{Dini}}$ Diffeomorphisms}

\begin{document}

\begin{abstract}
For $C^{1+\mathrm{Dini}}$ diffeomorphisms, we prove a Ledrappier--Young-type characterization of SRB measures among invariant measures whose supports admit dominated splittings. When the dominating bundle has only positive Lyapunov exponents, absolute continuity of the conditional measures on the corresponding Pesin unstable manifolds is equivalent to the partial Pesin entropy formula. As an application, we obtain SRB measures for partially hyperbolic mostly expanding attractors in the $C^{1+\mathrm{Dini}}$ category. Counterexamples are provided to show that these conclusions fail in the \(C^1\) category, even under uniform hyperbolicity.
\end{abstract}

\maketitle
\tableofcontents

\section{Introduction}\label{sec:introduction}

A central problem in smooth
ergodic theory is to study the distinguished invariant
measures.
The Lebesgue measure provides the natural reference distribution for initial
conditions, but is generally not preserved by the dynamics. One is
therefore led to look for invariant probability measures whose statistical
behavior can nevertheless be observed from a set of initial conditions of
positive Lebesgue measure. Measures with such a positive-volume basin are
usually called physical measures. In hyperbolic dynamics, the geometric
mechanism underlying this physical behavior is characterized by the
Sinai--Ruelle--Bowen property: after disintegrating the measure on
unstable manifolds, the conditional measures are absolutely
continuous with respect to the corresponding leafwise Riemannian volumes.

This picture originated in uniformly hyperbolic dynamics. For uniformly
hyperbolic attractors, Sinai, Ruelle, and Bowen constructed the canonical
invariant measures that now bear their initials; see
\cite{Sinai72,Bowen75,Ruelle76}. The passage from uniform to nonuniform
hyperbolicity is provided by Pesin theory. For a $C^{1+\alpha}$
diffeomorphism, local stable
and unstable manifolds exist at almost every point, thereby making it
possible to formulate the SRB property through the absolute continuity of conditional measures on
the measurable family of unstable manifolds; see \cite{BP07}.

There is a parallel description in terms of entropy and Lyapunov exponents.
For every invariant probability measure of a diffeomorphism, Margulis--Ruelle's inequality gives
\[
    h_\mu(f)
    \leqslant
    \sum_{\lambda_i(\mu,f)>0}\lambda_i(\mu,f),
\]
where the Lyapunov exponents are counted with multiplicity
\cite{Ruelle78}. The equality,
\[
    h_\mu(f)
    =
    \sum_{\lambda_i(\mu,f)>0}\lambda_i(\mu,f),
\]
is known as the Pesin entropy formula \cite{Pesin-entropyformula_1977}. The work of
Ledrappier--Young \cite{LY2,LY85}, building on
Ledrappier--Strelcyn \cite{LS82}, shows that in the classical
$C^{2}$ setting, this entropy identity is
equivalent to the geometric SRB property: the Pesin entropy formula holds
precisely when the conditional measures on unstable manifolds are
absolutely continuous with respect to leafwise volume. These results under $C^{1+\alpha}$ regularity were recently proved by Brown \cite{AaronBrown2022}, see also Saghin \cite{SaghinRadu2025}.

The regularity assumption behind this characterization is not merely
technical. For a general $C^1$
diffeomorphism, the modulus of continuity of the derivative may decay too
slowly for the equivalence between the SRB property and the Pesin entropy formula. The results in the presence of domination for
absolutely continuous measures were obtained by Sun and Tian
\cite{ST12}, while the results of Quas \cite{Quas99} and Avila--Bochi \cite{Avila06} demonstrate how fragile
absolute continuity phenomena can be in the $C^1$ category.

A natural regularity class between $C^1$ and $C^{1+\alpha}$ is the Dini condition. A
$C^1$ map $f$ is called $C^{1+\mathrm{Dini}}$ if the derivative $Df$ admits a modulus of
continuity $\omega$ satisfying
\[
    \int_0^1\frac{\omega(r)}{r}\,dr<\infty.
\]
Equivalently, the oscillation of the derivative is summable along every
geometric scale:
\[
    \sum_{n=1}^{\infty}\omega(\lambda^n)<\infty
    \qquad\text{for every }0<\lambda<1;
\]
see, for instance, \cite[Proposition~2.5]{Boukhecham2024}. The Dini
continuity could be strictly weaker than the H\"older continuity but retains exactly the summability needed
for bounded-distortion arguments. Dini regularity has appeared in
differentiable dynamics since the work of Anosov
\cite{Anosov67a,Anosov67b} and has subsequently been used in the study
of absolutely continuous invariant measures and SRB measures for
expanding and hyperbolic systems; see, among others,
\cite{LZ2000,FJ2001,JY2022,TXL2024, Boukhecham2024,Ounesli2024}.

This paper addresses whether the Ledrappier--Young characterization remains
valid for $C^{1+\mathrm{Dini}}$ diffeomorphisms admitting a dominated
splitting. We prove that, when all Lyapunov exponents along the dominating
bundle are positive, absolute continuity of the conditional measures on
the corresponding Pesin unstable manifolds is equivalent to the partial
Pesin entropy formula. We then apply this characterization to establish the
existence of SRB measures for $C^{1+\mathrm{Dini}}$ partially hyperbolic attractors with a mostly
expanding center, and conclude with uniformly hyperbolic $C^1$
counterexamples.

\subsection{The main theorem}

Let $M$ be a compact smooth Riemannian manifold without boundary, and let $f\in\mathrm{Diff}^1(M)$. 

If $K\subset M$ is a compact $f$-invariant set, a $Df$-invariant splitting
\[
    T_KM=E\oplus F
\]
is called a dominated splitting, with $F$ dominating $E$, if there exist constants $C\geqslant 1$ and $0<\tau<1$ such that
\[
    \bigl\|Df^n|_{E(x)}\bigr\|\,
    \bigl\|(Df^n|_{F(x)})^{-1}\bigr\|
    \leqslant C\tau^n
\]
for every $x\in K$ and every $n\geqslant1$. 

Let $\mu$ be an ergodic $f$-invariant probability measure and denote its Lyapunov exponents, given by the Oseledets Theorem \cite{Oseledets} as
\[
    \lambda_1(\mu,f)\geqslant\lambda_2(\mu,f)\geqslant\cdots
    \geqslant\lambda_{\dim M}(\mu,f),
\]
counted with multiplicity.

Assume that $T_{\operatorname{supp}(\mu)}M=E\oplus F$ is a dominated splitting with $F$ dominating $E$, and that $\dim F=k$. We assume that the exponents in $F$ are all positive. It follows that the exponents along $F$ are precisely $\lambda_1(\mu,f),\ldots,\lambda_k(\mu,f)$, and
\[
    \lambda_k(\mu,f)>\max\{\lambda_{k+1}(\mu,f),0\}.
\]
For $\mu$-almost every $x$, define
\begin{equation}\label{eq:def-Wk-intro}
    W^k(x)=\left\{y\in M:
        \limsup_{n\to\infty}\frac1n
        \log d(f^{-n}x,f^{-n}y)\leqslant-\lambda_k(\mu,f)
    \right\}.
\end{equation}
The results in \cite{ABC2011,Gan2002} identify $W^k(x)$ with the $k$-dimensional Pesin unstable manifold associated with the first $k$ Lyapunov exponents. In particular, $W^k(x)$ is tangent to $F(x)$ at $x$ for $\mu$-almost every $x$. It is worth noting that $W^k$ might not be a manifold without the assumption of dominated splitting, see \cite{BCS13} for instance.

A measurable partition $\xi$ is subordinate to $W^k$ if, for $\mu$-almost every $x$, the atom $\xi(x)$ is contained in $W^k(x)$ and contains a leafwise neighborhood of $x$. There exists an increasing measurable partition subordinate to $W^k$ \cite{TY,Part2,WWZ}; here increasing means that $f^{-1}\xi$ refines $\xi$. The corresponding $k$-th partial entropy is
\[
    h^k_\mu(f)=H_\mu(\xi\mid f\xi),
\]
and is independent of the choice of such a partition. Conditional measures on $W^k$ will always mean the Rokhlin conditional measures associated with a subordinate partition, regarded as measures on the corresponding Pesin unstable manifolds.

In this setting, the partial Ruelle inequality bounds $h^k_\mu(f)$ by the sum of the Lyapunov exponents carried by $F$ \cite{WWZ}. Our main result proves that the partial entropy formula holds if and only if the conditional measures on $W^k$ are absolutely continuous  for $C^{1+\mathrm{Dini}}$ diffeomorphisms.

\begin{maintheorem}\label{main-thm-intro}
Let $\mu$ be an ergodic invariant probability measure of $f\in\mathrm{Diff}^{1+\mathrm{Dini}}(M)$. Suppose that $T_{\operatorname{supp}(\mu)}M=E\oplus F$ is a dominated splitting, with $F$ dominating $E$, $\dim F=k$, and $\lambda_k(\mu,f)>0$. Then
\[
    h^k_\mu(f)=\sum_{i=1}^k\lambda_i(\mu,f)
\]
if and only if the conditional measures of $\mu$ on $W^k$ are absolutely continuous with respect to the leafwise Riemannian volume.
\end{maintheorem}

The implication from absolute continuity to the partial entropy formula is known for $C^1$ diffeomorphisms with dominated splitting \cite{WWZ}. The new content of this theorem is the converse under Dini regularity. 
The central issue is bounded distortion. Along a backward orbit in $W^k$, distances decay exponentially, so the variation of the derivative is governed by terms of the form $\omega(Ce^{-cn})$. The Dini condition is precisely what makes the resulting series summable.

To implement this observation, we first construct plaque families tangent to $F$ and choose their radii along typical backward orbits in a tempered way. A graph-transform argument then gives Dini control of the tangent distributions of these plaques. This control yields convergence and uniform bounds for the associated Jacobian-ratio products. Finally, following the Ledrappier--Young scheme, we use the limiting Jacobian cocycle to construct candidate conditional densities and show that equality in the partial entropy formula forces these densities to coincide with the Rokhlin disintegration of $\mu$. 

Next, we record the notion of the SRB measure used throughout the paper. Let $\mu$ be an $f$-invariant probability measure with a dominated splitting $T_{\operatorname{supp}(\mu)}M=E\oplus F$, where $F$ is the Oseledets bundle with respect to all positive Lyapunov exponents, and $E$ is the Oseledets bundle with respect to all non-positive Lyapunov exponents. Let $W^u$ denote the Pesin unstable manifolds (see \cite[Proposition 8.9]{ABC2011}). 

\begin{definition}\label{def:srb}
Under above assumptions, we call $\mu$ an SRB measure if its conditional measures with respect to a measurable partition subordinate to $W^u$ are absolutely continuous with respect to the leafwise Riemannian volume on $W^u$.
\end{definition}

\Cref{main-thm-intro} establishes the equivalence between this geometric definition and the corresponding entropy formula.

\subsection{Gibbs-{\it u} states}
We next apply 
\Cref{main-thm-intro}
to $C^{1+\mathrm{Dini}}$ partially
hyperbolic attractors. Let $U\subset M$ be a nonempty open set such that
$f(\overline U)\subset U$. The compact invariant set
\[
    \Lambda=\bigcap_{n\geqslant0}f^n(\overline U)
\]
is the attractor determined by the trapping region $U$. We say that
$\Lambda$ is partially hyperbolic if it admits a $Df$-invariant dominated
splitting
\[
    T_\Lambda M=E^s\oplus E^c\oplus E^u,
\]
where $E^s$ is uniformly contracting and $E^u$ is uniformly expanding.
We write $E^{cu}=E^c\oplus E^u$ and denote by $\mathcal F^u$ the strong
unstable foliation tangent to $E^u$.

Following Pesin and Sinai \cite{PesinSinai1982}, an invariant probability
measure supported on $\Lambda$ is called a Gibbs-$u$ state if its conditional
measures with respect to a measurable partition subordinate to $\mathcal F^u$
are absolutely continuous with respect to the corresponding leafwise
Riemannian volumes. We denote the space of Gibbs-$u$ states of $f$ on
$\Lambda$ by $\mathrm{Gibbs}^u(f,\Lambda)$. For the standard
properties of Gibbs-$u$ states in the $C^{1+\alpha}$ setting, see
\cite[Section~11.2]{BDV2005}; see also \cite{Dolgopyat2004Limit}.

Let $\mathcal M_f(\Lambda)$ denote the space of $f$-invariant Borel
probability measures supported on $\Lambda$. For
$\mu\in\mathcal M_f(\Lambda)$, write $u=\dim E^u$, and then $h^u_\mu(f)$ is exactly the unstable entropy associated with $\mathcal F^u$.
Combining \Cref{main-thm-intro} with the standard $C^1$ entropy arguments, see Crovisier--Yang--Zhang \cite{CrovisierYangZhang2020}, Yang \cite{Yang21}, and
Hua--Yang--Yang \cite{HuaYangYang2020} for instance,
we immediately obtain the following corollary.

\begin{corollary}\label{cor:gibbs-u-states}
Let $f\in\mathrm{Diff}^{1+\mathrm{Dini}}(M)$ and let $\Lambda$ be a
partially hyperbolic attractor. Then $\mathrm{Gibbs}^u(f,\Lambda)$ is a
nonempty weak-$*$ compact convex subset of $\mathcal M_f(\Lambda)$,
and the extreme points
of $\mathrm{Gibbs}^u(f,\Lambda)$ are precisely the ergodic Gibbs-$u$ states.
Moreover, for every $\mu\in\mathcal M_f(\Lambda)$,
\[
    \mu\in\mathrm{Gibbs}^u(f,\Lambda)
    \quad\Longleftrightarrow\quad
    h^u_\mu(f)
    =
    \int_\Lambda
    \log\left|\det\left(Df_x|_{E^u(x)}\right)\right|\,d\mu(x).
\]
\end{corollary}
\subsection{Mostly expanding attractors}
In this section, we establish the existence of the SRB measures for $C^{1+\mathrm{Dini}}$ mostly expanding partially hyperbolic attractors

There are two related notions in the literature under the name ``mostly expanding.'' In the original work of Alves--Bonatti--Viana \cite{ABV}, the hypothesis is formulated for a splitting $E^{ss}\oplus E^{cu}$, with $E^{ss}$ uniformly contracting, and requires nonuniform expansion along $E^{cu}$ on a set of positive Lebesgue measure. Here we use the Gibbs-$u$ formulation for partially hyperbolic systems, as in \cite{AnderssonVasquez2018,AnderssonVasquez2020}.

We say that the center is mostly expanding in the Gibbs-$u$ sense if, for every Gibbs-$u$ state $\nu$, all Lyapunov exponents along $E^c$ are positive at $\nu$-almost every point.  A standard compactness argument for the smallest center exponent gives the following uniform consequence: there exist $N_0\in\mathbb N$ and $c_0>0$ such that
\[
    \int \log m(Df^{N_0}|_{E^c})\,d\nu\geqslant c_0
\]
for every Gibbs-$u$ state $\nu$ of $f|_\Lambda$, where
\[
    m(A)=\inf_{\|v\|=1}\|Av\|
\]
denotes the conorm; see \cite[Lemma~4.1]{AnderssonVasquez2018}. This formulation is $C^1$-open; see \cite{AnderssonVasquez2020,Yang21,Yang2023MostlyExpanding}.

\begin{corollary}\label{coro2}
Let $f\in\mathrm{Diff}^{1+\mathrm{Dini}}(M)$ and let $\Lambda$ be a partially hyperbolic attractor whose center is mostly expanding in the Gibbs-$u$ sense. Then $f$ admits an ergodic SRB measure supported on $\Lambda$.
\end{corollary}

The proof combines smooth approximation with the upper semicontinuity of partial entropy \cite{TY2}. We approximate $f$ in the $C^1$ topology by smooth diffeomorphisms whose continuations of $\Lambda$ remain mostly expanding, and choose an ergodic SRB measure for each approximating system using the smooth theory \cite{ABV,AnderssonVasquez2018}. A weak-$*$ limit satisfies the entropy formula along $E^u$ and is therefore a Gibbs-$u$ state by \Cref{main-thm-intro}. The mostly expanding center then gives uniform positivity along $E^c$, while the upper semicontinuity of partial entropy transfers the center-unstable entropy formula to the limit. Applying \Cref{main-thm-intro} to the dominated bundle $E^{cu}$ completes the proof.

Related existence and stability results for partially hyperbolic attractors, including mostly contracting and mostly non-expanding mechanisms. See, for instance, \cite{BV2000,QiuHao,ZERONOISE}. The methods developed here can also be adapted to these neighboring settings.

The paper is organized as follows. \Cref{sec:dini} develops the elementary calculus of Dini moduli of bundle homomorphisms. \Cref{sec:geometric} constructs the plaque families and proves the bounded-distortion estimates. \Cref{sec:proof-main} proves \Cref{main-thm-intro}. \Cref{sec:application} establishes \Cref{coro2}. Finally, \Cref{sec:c1-example} presents the $C^1$ counterexamples.

\section{Dini continuity}\label{sec:dini}

\subsection{Definition of Dini continuity}

\begin{definition}
A function $\omega\colon [0,+\infty)\to[0,+\infty)$ is a modulus of continuity if it is non-decreasing, continuous, and concave, with $\omega(0)=0$ and $\omega(r)>0$ for $r>0$. A modulus of continuity $\omega$ is called Dini if
\[
    \int_0^1\frac{\omega(r)}{r}\,dr<\infty.
\]
\end{definition}

\begin{lemma}\label{lem:basic-modulus}
Let $\omega$ be a modulus of continuity. Then, for every $r\geqslant0$,
\[
\begin{array}{lll}
    t\omega(r)\leqslant \omega(tr)\leqslant \omega(r), & & 0\leqslant t\leqslant 1,\\[2mm]
    \omega(r)\leqslant \omega(tr)\leqslant t\omega(r), & & t\geqslant 1.
\end{array}
\]
Consequently,
\[
    \min\{1,t\}\omega(r)\leqslant \omega(tr)\leqslant \max\{1,t\}\omega(r),
    \qquad t,r\geqslant0.
\]
\end{lemma}

\begin{proof}
Since $\omega$ is concave and $\omega(0)=0$, the function $r\mapsto \omega(r)/r$ is non-increasing on $(0,\infty)$. The two displayed estimates follow immediately from this fact and from the monotonicity of $\omega$.
\end{proof}

Let $(X,d)$ and $(Y,\rho)$ be metric spaces. For a map $\phi\colon X\to Y$, define
\[
    \omega_\phi(r)=\sup_{d(x_1,x_2)\leqslant r}\rho(\phi(x_1),\phi(x_2)).
\]

\begin{definition}
A map $\phi\colon X\to Y$ is Dini-continuous if there exists a Dini modulus of continuity $\omega$ such that $\omega_\phi(r)\leqslant \omega(r)$ for all $r\geqslant0$.

Given a modulus of continuity $\omega$, set
\[
    C^\omega(X,Y)=\left\{\phi\colon X\to Y:
    \begin{array}{l}
    \text{there exists } C\geqslant0 \text{ such that}\\
    \omega_\phi(r)\leqslant C\omega(r) \text{ for all } r\geqslant0
    \end{array}\right\}.
\]
For $\phi\in C^\omega(X,Y)$, define
\[
    [\phi]_\omega=\inf\{C\geqslant0: \omega_\phi(r)\leqslant C\omega(r)\text{ for all }r\geqslant0\}.
\]
We also write
\[
    C^{\mathrm{Dini}}(X,Y)=\bigcup_\omega C^\omega(X,Y),
\]
where the union is taken over all Dini moduli $\omega$.
\end{definition}

\begin{remark}
Only the behavior of the modulus near $0$ is relevant.
\end{remark}
\begin{remark}
The H\"older continuity implies the Dini continuity, since $\omega(r)=r^\alpha$ is Dini for every $\alpha>0$.
\end{remark}
\begin{proposition}\label{prop:dini-summability}
Let $\omega$ be a Dini modulus of continuity. Then, for every $0<\lambda<1$,
\[
    \sum_{i=1}^{\infty}\omega(\lambda^i)<\infty.
\]
\end{proposition}

\begin{proof}
Since $\omega$ is non-decreasing, for every $i\geqslant1$,
\[
    \omega(\lambda^i)
    \leqslant \frac{1}{1-\lambda}\int_{\lambda^i}^{\lambda^{i-1}}\frac{\omega(r)}{r}\,dr.
\]
Summing over $i$ gives the result from the Dini condition.
\end{proof}

\begin{remark}
The convergence of $\sum_i\omega(\lambda^i)$ is equivalent to the Dini condition. This is the summability property used in the distortion estimates below. See \cite[Proposition~2.5]{Boukhecham2024}, \cite{FJ2001}, or \cite{LZ2000} for instance.
\end{remark}

\subsection{Modulus of continuity for bundle homomorphisms}

We now define $C^{1+\mathrm{Dini}}$ maps on manifolds. The definition is analogous to that of $C^{1+\alpha}$ maps. We follow the approach of \cite{Anosov67b,Sarig2013}.

Let $M$ and $M'$ be smooth Riemannian manifolds without boundary, and assume that $M$ is compact. Let $f\colon M\to M'$ be a continuous map. Let $\pi\colon \mathscr E\to M$ and $\pi'\colon \mathscr E'\to M'$ be smooth vector bundles equipped with inner products, and let $\phi\colon \mathscr E\to\mathscr E'$ be a continuous bundle homomorphism covering $f$; that is,
\[
    \pi'\circ\phi=f\circ\pi,
\]
and $\phi_x=\phi|_{\mathscr E_x}\colon\mathscr E_x\to\mathscr E'_{fx}$ is linear for every $x\in M$. We consider $\phi$ from a different point of view. Define the vector bundle $L(\mathscr E,\mathscr E',f)$ over $M$ whose fiber at $x\in M$ consists of the linear maps from $\mathscr E_x$ to $\mathscr E'_{fx}$. Then $\phi$ is naturally a section of $L(\mathscr E,\mathscr E',f)$. To define its modulus of continuity, we must compare $\phi_x$ and $\phi_y$ for different points $x,y\in M$, even though these maps have different domains and ranges.

By local triviality, every $x\in M$ has an open neighborhood $D$ and a smooth fiber-coordinate map
\[
    \theta_D\colon \mathscr E|_D\to\mathbb R^{\dim\mathscr E}
\]
whose restriction to each fiber is a linear isometry. Let $\mathscr D$ be a finite cover of $M$ by such neighborhoods, and let $\varepsilon(\mathscr D)$ be a Lebesgue number of $\mathscr D$. Thus, if $d(x,y)\leqslant\varepsilon(\mathscr D)$, then $x$ and $y$ lie in a common element of $\mathscr D$. Similarly, choose fiber-coordinate maps $\theta_{D'}$ and a finite open cover $\mathscr D'$ of $f(M)\subset M'$, with a Lebesgue number $\varepsilon(\mathscr D')$. Choose $\varepsilon'>0$ so small that
\[
    d(x,y)\leqslant\varepsilon'
    \quad\Longrightarrow\quad
    d(fx,fy)\leqslant\varepsilon(\mathscr D').
\]
Then $fx$ and $fy$ lie in a common element of $\mathscr D'$. Set
\[
    \varepsilon=\min\{\varepsilon(\mathscr D),\varepsilon'\}.
\]
For $z\in D$, write $\theta_{D,z}=\theta_D|_{\mathscr E_z}$, and use the analogous notation for the target bundle. Instead of comparing $\phi_x$ directly with $\phi_y$, we compare the Euclidean linear maps
\[
    \theta_{D',fx}\circ\phi_x\circ\theta_{D,x}^{-1}
    \quad\text{and}\quad
    \theta_{D',fy}\circ\phi_y\circ\theta_{D,y}^{-1}.
\]

\begin{definition}
Let $\phi\colon\mathscr E\to\mathscr E'$ be a bundle homomorphism as above. For $0\leqslant r\leqslant\varepsilon$, define its modulus of continuity by
\[
    \omega_\phi(r)
    \coloneqq
    \sup_{\substack{D\in\mathscr D,\ D'\in\mathscr D'\\
                     x,y\in D,\ fx,fy\in D'\\
                     d(x,y)\leqslant r}}
    \left\|
        \theta_{D',fx}\circ\phi_x\circ\theta_{D,x}^{-1}
        -
        \theta_{D',fy}\circ\phi_y\circ\theta_{D,y}^{-1}
    \right\|.
\]
For $r\geqslant\varepsilon$, extend this definition by setting
\[
    \omega_\phi(r)=\omega_\phi(\varepsilon).
\]

Given a modulus of continuity $\omega$, we say that $\phi$ is $\omega$-continuous if there exists $C\geqslant0$ such that
\[
    \omega_\phi(r)\leqslant C\omega(r),
    \qquad r\geqslant0.
\]
If $\phi$ is $\omega$-continuous for some Dini modulus $\omega$, then we say that $\phi$ is Dini-continuous.

Given a modulus of continuity $\omega$, define
\[
\begin{aligned}
    C^{1+\omega}(M,M')
    \coloneqq
    \bigl\{f\in C^1(M,M'):\;&
    Df\colon TM\to TM'
    \text{is }\omega\text{-continuous}\bigr\}.
\end{aligned}
\]
and
\[
    C^{1+\mathrm{Dini}}(M,M')
    =\bigcup_\omega C^{1+\omega}(M,M'),
\]
where the union is taken over all Dini moduli $\omega$. The spaces $\mathrm{Diff}^{1+\omega}(M)$ and $\mathrm{Diff}^{1+\mathrm{Dini}}(M)$ are defined analogously.
\end{definition}

\begin{remark}
\begin{enumerate}[label=(\roman*),leftmargin=*]
    \item The chosen fiber-coordinate maps identify nearby fibers and hence give a local coordinate comparison for the section $\phi$. For $r\leqslant\varepsilon$, the function $\omega_\phi(r)$ is the corresponding coordinate modulus.

    \item The numerical function $\omega_\phi$ depends on the chosen fiber coordinates and need not be equivalent, up to pure multiplicative constants, to the modulus obtained from another choice. More precisely, let $\widetilde\omega_\phi$ be defined using another pair of finite smooth trivializing families. Then there exists $C\geqslant1$ such that, for all sufficiently small $r$,
    \[
    \begin{aligned}
        \widetilde\omega_\phi(r)
        &\leqslant
        C\bigl(\omega_\phi(r)+\max\{\|\phi\|,1\}\cdot(r+\omega_f(r))\bigr),\\
        \omega_\phi(r)
        &\leqslant
        C\bigl(\widetilde\omega_\phi(r)+\max\{\|\phi\|,1\}\cdot(r+\omega_f(r))\bigr),
    \end{aligned}
    \]
    where $\|\phi\|=\sup_{x\in M}\|\phi_x\|$. Indeed, after passing to a common finite refinement, the two local matrix representations satisfy
    \[
        \widetilde\Phi(z)=Q(fz)\Phi(z)P(z)^{-1},
    \]
    where $P$ and $Q$ are smooth transition matrices whose norms, inverse norms, and Lipschitz constants are uniformly bounded. Expanding $\widetilde\Phi(x)-\widetilde\Phi(y)$ gives the first estimate, and interchanging the two coordinate systems gives the second. See Lemma \ref{lem:bundle-Dini-ineq} below for more details.

    If $f$ is Lipschitz, then $\omega_f(r)\leqslant Lr$ for some $L>0$. Moreover, by \Cref{lem:basic-modulus}, every modulus $\omega$ satisfies
    \[
        \omega(r)\geqslant\omega(1)r,
        \qquad 0\leqslant r\leqslant1.
    \]
    Hence, after enlarging the multiplicative constant to cover the remaining values of $r$, $\phi$ is $\omega$-continuous for one choice of fiber coordinates if and only if it is $\omega$-continuous for the other. In particular, the definitions of $C^{1+\omega}$ and $C^{1+\mathrm{Dini}}$ are independent of the chosen trivializations, since every $C^1$ map on the compact manifold $M$ is Lipschitz.

    \item The same argument applies when the fiber-coordinate maps are smooth fiberwise linear isomorphisms rather than isometries; only the uniform constant $C$ changes. Choosing isometries simplifies the estimates below.
\end{enumerate}
\end{remark}

\subsection{Some useful inequalities}

The following two inequalities are adapted from \cite[p.~821]{Anosov67b}.

\begin{lemma}\label{lem:bundle-Dini-ineq}
Let $\mathscr E$, $\mathscr E'$, and $\mathscr E''$ be vector bundles over $M$, $M'$, and $M''$, respectively. Assume $M$ is compact. Let
\[
    \phi\colon\mathscr E\to\mathscr E',
    \qquad
    \psi\colon\mathscr E'\to\mathscr E''
\]
be bundle homomorphisms covering $f\colon M\to M'$ and $g\colon M'\to M''$, respectively. Then, for all sufficiently small $r$, the following estimates hold:
\begin{enumerate}[label=(D\arabic*),leftmargin=*]
    \item\label{D1}
    \[
        \omega_{\psi\cdot\phi}(r)
        \leqslant
        \omega_\psi\bigl(\omega_f(r)\bigr)\,\|\phi\|
        +\|\psi\|\,\omega_\phi(r).
    \]

    \item\label{D2} 
    \[
        \omega_{\phi^{-1}}(r)
        \leqslant
        \frac{1}{m(\phi)^2}\,
        \omega_\phi\bigl(\omega_{f^{-1}}(r)\bigr).\qquad\quad
    \]
    when $\phi$ is invertible and \(
        m(\phi)\coloneqq\|\phi^{-1}\|^{-1}.
    \)
\end{enumerate}
\end{lemma}

\begin{proof}
Let $\mathscr D$, $\mathscr D'$, and $\mathscr D''$ be finite trivializing covers of $M$, $f(M)$, and $g(f(M))$, respectively. Choose $\varepsilon>0$ so small that, whenever $d(x,y)\leqslant\varepsilon$, the points $x,y$ lie in a common element $D\in\mathscr D$, the points $fx,fy$ lie in a common element $D'\in\mathscr D'$, and the points $g(fx),g(fy)$ lie in a common element $D''\in\mathscr D''$.

Fix $r\in[0,\varepsilon]$ and points $x,y$ with $d(x,y)\leqslant r$. For simplicity, write
\[
    \theta=\theta_D,
    \qquad
    \theta'=\theta_{D'},
    \qquad
    \theta''=\theta_{D''},
\]
and denote the restrictions to the corresponding fibers by
\[
    \theta_x,\ \theta_y,\ \theta'_{fx},\ \theta'_{fy},\
    \theta''_{g(fx)},\ \theta''_{g(fy)}.
\]
In Euclidean coordinates, the difference between
$(\psi\phi)_x=\psi_{fx}\phi_x$ and
$(\psi\phi)_y=\psi_{fy}\phi_y$ can be decomposed as follows:
\begin{align*}
&\theta''_{g(fx)}\circ\psi_{fx}\circ\phi_x\circ\theta_x^{-1}
-
\theta''_{g(fy)}\circ\psi_{fy}\circ\phi_y\circ\theta_y^{-1}
\\
&=\Bigl(
    \theta''_{g(fx)}\circ\psi_{fx}\circ(\theta'_{fx})^{-1}
    -
    \theta''_{g(fy)}\circ\psi_{fy}\circ(\theta'_{fy})^{-1}
  \Bigr)
  \circ\theta'_{fx}\circ\phi_x\circ\theta_x^{-1}
\\
&\quad+
  \theta''_{g(fy)}\circ\psi_{fy}\circ(\theta'_{fy})^{-1}
  \circ\Bigl(
      \theta'_{fx}\circ\phi_x\circ\theta_x^{-1}
      -
      \theta'_{fy}\circ\phi_y\circ\theta_y^{-1}
  \Bigr)
\\
&=:I+II.
\end{align*}
Because the fiber-coordinate maps are isometries,
\begin{align*}
\|I\|
&\leqslant
\left\|
    \theta''_{g(fx)}\circ\psi_{fx}\circ(\theta'_{fx})^{-1}
    -
    \theta''_{g(fy)}\circ\psi_{fy}\circ(\theta'_{fy})^{-1}
\right\|
\left\|\theta'_{fx}\circ\phi_x\circ\theta_x^{-1}\right\|
\\
&\leqslant \omega_\psi\bigl(d(fx,fy)\bigr)\,\|\phi_x\|
\\
&\leqslant \omega_\psi\bigl(\omega_f(r)\bigr)\,\|\phi\|.
\end{align*}
Similarly,
\begin{align*}
\|II\|
&\leqslant
\left\|
    \theta''_{g(fy)}\circ\psi_{fy}\circ(\theta'_{fy})^{-1}
\right\|
\left\|
    \theta'_{fx}\circ\phi_x\circ\theta_x^{-1}
    -
    \theta'_{fy}\circ\phi_y\circ\theta_y^{-1}
\right\|
\\
&\leqslant \|\psi\|\,\omega_\phi(r).
\end{align*}
Combining the estimates for $I$ and $II$ proves \Cref{D1}.

The estimate in \Cref{D2} follows from the identity
\[
    A^{-1}-B^{-1}=B^{-1}(B-A)A^{-1}.
\]
See \cite[p.~821]{Anosov67b} for further details.
\end{proof}

\begin{remark}
The proof of \Cref{D1} gives a more convenient estimate
\[
\begin{aligned}
    \omega_{\psi\cdot\phi}(r)
    &\leqslant
    \omega_{\psi\circ f}(r)\,\|\phi\|
    +\|\psi\|\,\omega_\phi(r)
    \\
    &\leqslant
    \omega_\psi\bigl(\omega_f(r)\bigr)\,\|\phi\|
    +\|\psi\|\,\omega_\phi(r).
\end{aligned}
\]
Here $\psi\circ f$ is regarded as a section of a vector bundle $\widetilde{\mathscr E}$ over $M$, whose fiber at $x\in M$ is the space of linear maps from $\mathscr E'_{fx}$ to $\mathscr E''_{g(fx)}$. Similarly, \Cref{D2} has a more convenient form
\[
    \omega_{\phi^{-1}}(r)
    \leqslant
    \frac{1}{m(\phi)^2}\,\omega_{\phi\circ f^{-1}}(r)
    \leqslant
    \frac{1}{m(\phi)^2}\,
    \omega_\phi\bigl(\omega_{f^{-1}}(r)\bigr).
\]

So, when $f$ and $g$ are diffeomorphisms and $\phi=Df, \psi=Dg$, (D1) and (D2) read as follows.

$$\omega_{D(g\circ f)}(r)\leqslant\omega_{Dg\circ f}(r)\|Df\|+\|Dg\|\omega_{Df}(r).\leqno (D1)$$

$$\omega_{D(f^{-1})}(r)\leqslant\frac{1}{m(Df)^2}\omega_{Df\circ f^{-1}}(r).\leqno (D2)$$

\end{remark}

\begin{corollary}\label{cor:iteration}
Let $f\in\mathrm{Diff}^{1+\omega}(M)$. Then, for every $n\in\mathbb Z$,
\(
    f^n\in\mathrm{Diff}^{1+\omega}(M).
\)
\end{corollary}

\begin{remark}
In the remainder of the paper, we will mainly estimate moduli of continuity in local coordinates.
\end{remark}

\section{Geometric structures}\label{sec:geometric}

This section is the technical core of the paper. We construct local plaques tangent to the dominated bundle $F$, control their Dini regularity along typical backward orbits by graph transformation methods, and derive the bounded distortion estimate needed in the proof of the main theorem.

Throughout this section, we assume the hypotheses of \Cref{main-thm-intro}. Thus, $\mu$ is an ergodic invariant probability measure of $f\in\mathrm{Diff}^{1+\omega}(M)$ for some Dini modulus $\omega$, and
\[
    T_{\operatorname{supp}(\mu)}M=E\oplus F,
    \qquad \dim F=k,
\]
is a dominated splitting, with $F$ dominating $E$ and $\lambda_k(\mu,f)>0$. We write $\lambda_k=\lambda_k(\mu,f)$ and fix a number $\theta$ with
\[
    \max\{\lambda_{k+1}(\mu,f),0\}<\theta<\lambda_k(\mu,f).
\]

\subsection{Dominated splitting and a convenient iterate}

We use an adapted Riemannian metric. According to Gourmelon \cite{Gourmelon2007}, there exists $0<\lambda<1$ such that for every $x\in\operatorname{supp}(\mu)$ and every pair of unit vectors $v\in E(x)$ and $w\in F(x)$,
\[
    \frac{\|Df_xv\|}{\|Df_xw\|}\leqslant \lambda.
\]

The following consequence of the Oseledets' theorem is taken from \cite[Lemma 8.4]{ABC2011}.

\begin{lemma}\label{lem:ABC-average}
For every $\eta>0$, there is an integer $N=N(\eta)$ such that, for $\mu$-almost every $x$ and every $n\geqslant N$, the averages
\[
    \frac{1}{\ell n}\sum_{i=1}^{\ell}\log m(Df^n|_{F(f^{-in}x)})
\]
converge, as $\ell\to\infty$, to a number in $(\lambda_k-\eta,\lambda_k]$.
\end{lemma}

Fix $\eta>0$ so small that $\theta<\lambda_k-\eta$, and choose $N=N(\eta)$. Set
\[
    g=f^N.
\]
Then, for $\mu$-almost every $x$,
\begin{equation}\label{eq:Birkhoff-positive-F}
    \lim_{\ell\to\infty}\frac1{\ell N}\sum_{i=1}^{\ell}
    \log m(Dg|_{F(g^{-i}x)})>\lambda_k-\eta.
\end{equation}
The same splitting $E\oplus F$ is also a dominated splitting for $g$, and $g\in\mathrm{Diff}^{1+\omega}(M)$ by \Cref{cor:iteration}.

We use one fixed Dini modulus throughout the proof.

\begin{lemma}\label{lem:fixed-modulus}
There exists $C_0\geqslant1$ such that, for all sufficiently small $r$,
\[
    \omega_{Dg}(r)\leqslant C_0\omega(r),
    \qquad
    \omega(r)\geqslant r/C_0.
\]
\end{lemma}

\begin{proof}
Since $g\in\mathrm{Diff}^{1+\omega}(M)$, there exists $C_g\geqslant0$ such that
\[
    \omega_{Dg}(r)\leqslant C_g\omega(r)
\]
for all sufficiently small $r$. On the other hand, \Cref{lem:basic-modulus}, applied with base radius $1$ and factor $r$, gives
\[
    \omega(r)\geqslant r\omega(1),
    \qquad 0\leqslant r\leqslant1.
\]
Thus, the conclusion follows after taking
\[
    C_0=\max\left\{1,C_g,\frac{1}{\omega(1)}\right\}.\qedhere
\]
\end{proof}

\subsection{Modulus of continuity in local charts}
In this section, we construct local charts and bumped derivative sequences. For similar discussions, see \cite{TY} for instance.
Choose $r_0\in(0,1)$ so small that, for every $x\in M$,
\[
    g_x=\exp_{gx}^{-1}\circ g\circ\exp_x
\]
is well-defined on $T_xM(2r_0)$. Since the exponential maps are smooth, \Cref{lem:bundle-Dini-ineq} and \Cref{lem:basic-modulus} give the following uniform estimate.

\begin{lemma}\label{lem:local-Dini-Dg}
After decreasing $r_0$ if necessary, there is a constant $C_1\geqslant C_0$ such that
\[
    \omega_{Dg_x}(r)\leqslant C_1\omega(r),
    \qquad 0\leqslant r\leqslant 2r_0,
\]
for every $x\in M$.
\end{lemma}

\begin{proof}
For $v\in T_xM(2r_0)$,
\[
    Dg_x(v)=D(\exp_{gx}^{-1})_{g(\exp_x v)}\circ Dg_{\exp_x v}\circ D(\exp_x)_v.
\]
After decreasing $r_0$, the norms and Lipschitz constants of the exponential factors are uniformly bounded. Moreover, the maps $\exp_x$ and $g\circ\exp_x$ are uniformly Lipschitz. Hence, the moduli of the two exponential factors are bounded by $Cr$, while the middle factor has a modulus bounded by $C_0\omega(Cr)$. By \Cref{lem:fixed-modulus,lem:basic-modulus}, all three moduli are bounded by a uniform constant times $\omega(r)$. Applying \Cref{lem:bundle-Dini-ineq} twice and absorbing the uniform factors into the constant gives the estimate.
\end{proof}

\begin{lemma}\label{lem:local-derivative-small}
For every $\varepsilon>0$, there exists $\rho\in(0,r_0]$ such that, for every $x\in M$ and every $v\in T_xM(2\rho)$,
\[
    \|Dg_x(v)-Dg_x(0)\|\leqslant \varepsilon.
\]
\end{lemma}

\begin{proof}

This follows from the uniform continuity of $Dg_x$ in exponential charts; the uniformity in $x$ follows from the compactness of $M$. See \cite[Lemma~4.10]{WL}.
\end{proof}

Write $A_x=Dg|_{E(x)}$ and $B_x=Dg|_{F(x)}$. Since the splitting is dominated and the support is compact, we may choose $\varepsilon_0>0$ sufficiently small so that, for all $x\in\operatorname{supp}(\mu)$,
\begin{equation}\label{eq:epsilon0-choice}
    0<\varepsilon_0<\frac13\bigl(m(B_x)-\|A_x\|\bigr).
\end{equation}
Let $\alpha\colon\mathbb R\to[0,1]$ be a $C^\infty$ function with $\alpha(t)=1$ for $|t|\leqslant1$ and $\alpha(t)=0$ for $|t|\geqslant2$. Set $C_\alpha=2\|\alpha'\|+1$. In what follows, $0<\varepsilon\leqslant\varepsilon_0$, and $\rho=\rho(\varepsilon/C_\alpha)$ is chosen from \Cref{lem:local-derivative-small}. With this choice, all perturbation terms  are bounded by $\varepsilon$.

Put
\[
    \beta(v)=\alpha(\|v\|/\rho),\qquad v\in T_xM.
\]
Define a map $\widetilde g_x\colon T_xM\to T_{gx}M$ by
\begin{equation}\label{eq:truncated-map}
    \widetilde g_x(v)=\beta(v)g_x(v)+(1-\beta(v))Dg_x(0)v.
\end{equation}
Then $\widetilde g_x=g_x$ on $T_xM(\rho)$ and $D\widetilde g_x(0)=Dg_x(0)=Dg(x)$. Write
\[
    \widetilde g_x(v)=Dg_x(0)v+h_x(v).
\]

\begin{lemma}\label{lem:h-small}
For every $x\in M$,
\(
  \  h_x(0)=0,
    \quad
    \|Dh_x\|\leqslant \varepsilon.
\)
\end{lemma}
\begin{proof}
It is direct that
\(
    h_x(0)=0.
\)

Fix $x\in M$ and put
\(
    \eta=\frac{\varepsilon}{C_\alpha}.
\)
We first estimate the derivative of the bump function. Choose an orthonormal
basis of $T_xM$ and write $v=(v_1,\ldots,v_{\dim M})$. If $v\neq0$, then
\[
    \frac{\partial\beta}{\partial v_i}(v)
    =\alpha'\!\left(\frac{\|v\|}{\rho}\right)
      \frac{v_i}{\rho\|v\|},
    \qquad 1\leqslant i\leqslant\dim M.
\]
Consequently, it follows that
\[
    \|D\beta\|\leqslant\frac{\|\alpha'\|}{\rho}.
\]

By the chain rule, we have
$$Dh_x(v)=(g_x(v)-Dg_x(0)v)D\beta(v)+\beta(v)(Dg_x(v)-Dg_x(0)).$$
Therefore, for $\|v\|\leqslant 2\rho,$
\begin{equation*}
    \begin{aligned}
        \|Dh_x(v)\| &\leqslant \|D\beta(v)\|\cdot \|g_x(v)-Dg_x(0)v\| + \|\beta(v)\|\cdot \|Dg_x(v)-Dg_x(0)\|\\
        & \leqslant \|D\beta\|\cdot 2\rho \cdot \max_{\|v\|\leqslant 2\rho}\|Dg_x(v)-Dg_x(0)\| + 1 \cdot \max_{\|v\|\leqslant 2\rho}\|Dg_x(v)-Dg_x(0)\| \\
        &\leqslant (2\rho\|D\beta\|+1) \cdot\max_{\|v\|\leqslant 2\rho}\|Dg_x(v)-Dg_x(0)\|\\
        &\leqslant (2\rho \|D\beta \|+1) \eta\\
        &\leqslant (2 \|\alpha'\| +1) \eta=\varepsilon.
    \end{aligned}
\end{equation*}    
Since $Dh_x(v)$ is non-vanishing only inside $\{\|v\|\leqslant 2\rho\}$, we have that for any $v$, $\|Dh_x(v)\|\leqslant\varepsilon.$ 
Then the proof is complete.
\end{proof}

\begin{lemma}\label{lem:Dh-Dini}
For every $0<\varepsilon\leqslant\varepsilon_0$, there exists $C_2\geqslant C_1$ such that, for every $x\in M$,
\[
    \omega_{Dh_x}(r)\leqslant C_2\omega(r),
    \qquad r\geqslant0.
\]
\end{lemma}

\begin{proof}
If $\|v-w\|\geqslant \rho$, then \Cref{lem:h-small} and the monotonicity of $\omega$ give
\[
    \|Dh_x(v)-Dh_x(w)\|
    \leqslant 2\varepsilon
    \leqslant \frac{2\varepsilon}{\omega(\rho)}\,
    \omega(\|v-w\|).
\]
Now assume that $\|v-w\|\leqslant \rho$, and $\|v\|\leqslant\|w\|$ without loss of generality. We have the following three cases.

\noindent {\bf Case 1. } $\|v\|\geqslant 2\rho, \|w\|\geqslant 2\rho$. It follows that $Dh_x(v)-Dh_x(w)=0$. 

\noindent {\bf Case 2. }
$\|v\|\leqslant 2\rho, \|w\|\geqslant 2\rho$. Then $Dh_x(w)=0$.  And we have
\begin{eqnarray*}
\|Dh_x(v)-Dh_x(w)\|&=&\|Dh_x(v)\|\\
&=& \|(g_x(v)-Dg_x(0)v)D\beta(v)+\beta(v)(Dg_x(v)-Dg_x(0))\|\\
&\leqslant& \|(g_x(v)-Dg_x(0)v)(D\beta(v)-D\beta(w))\|+\\
& & \|(\beta(v)-\beta(w))(Dg_x(v)-Dg_x(0))\|\\
&\leqslant& (2\rho\frac{\varepsilon}{C_\alpha}\|D^2\beta\|+\frac{\varepsilon}{C_\alpha}\|D\beta\|)\|v-w\|.
\end{eqnarray*}

\noindent {\bf Case 3. } $\|v\|\leqslant 2\rho, \|w\|\leqslant 2\rho$.
\begin{eqnarray*}
\|Dh_x(v)-Dh_x(w)\|
&\leqslant& \|(g_x(v)-Dg_x(0)v)D\beta(v)-(g_x(w)-Dg_x(0)w)D\beta(w)\| +\\
& & \|\beta(v)(Dg_x(v)-Dg_x(0))-\beta(w)(Dg_x(w)-Dg_x(0))\|\\
&\leqslant& \|(g_x(v)-Dg_x(0)v)(D\beta(v)-D\beta(w))\|+\\
& & \|g_x(v)-g_x(w)-Dg_x(0)(v-w)\|\|D\beta(w)\|+\\
& & \|\beta(v)(Dg_x(v)-Dg_x(w))\|+\\
& & \|(\beta(v)-\beta(w))(Dg_x(w)-Dg_x(0))\|\\
&\leqslant& (2\rho\frac{\varepsilon}{C_\alpha}\|D^2\beta\|+2\frac{\varepsilon}{C_\alpha}\|D\beta\|)\|v-w\|+\omega_{Dg_x}(\|v-w\|).
\end{eqnarray*}
Thus, if $\|v-w\|\leqslant\rho$, then \Cref{lem:fixed-modulus,lem:local-Dini-Dg} gives
\[
\begin{aligned}
    \|Dh_x(v)-Dh_x(w)\|
    &\leqslant
    \left(2\rho\frac{\varepsilon}{C_\alpha}\|D^2\beta\|
    +2\frac{\varepsilon}{C_\alpha}\|D\beta\|\right)\|v-w\|
    +\omega_{Dg_x}(\|v-w\|) \\
    &\leqslant
    \left[
        C_0\left(2\rho\frac{\varepsilon}{C_\alpha}\|D^2\beta\|
        +2\frac{\varepsilon}{C_\alpha}\|D\beta\|\right)
        +C_1
    \right]\omega(\|v-w\|).
\end{aligned}
\]
Combining this estimate with the case $\|v-w\|\geqslant\rho$, we obtain
\[
    \omega_{Dh_x}(r)\leqslant C_2\omega(r),
    \qquad r\geqslant0,
\]
where
\[
    C_2=
    \max\left\{
        C_1+C_0\left(2\rho\frac{\varepsilon}{C_\alpha}\|D^2\beta\|
        +2\frac{\varepsilon}{C_\alpha}\|D\beta\|\right),
        \frac{2\varepsilon}{\omega(\rho)}
    \right\}.\qedhere
\]
\end{proof}

Using the splitting $T_xM=E(x)\oplus F(x)$, write
\[
    \widetilde g_x(u,v)=\bigl(A_xu+h_{x,1}(u,v),\,B_xv+h_{x,2}(u,v)\bigr),
    \qquad u\in E(x),\ v\in F(x).
\]
The angles between $E$ and $F$ are uniformly bounded away from zero. 
\begin{remark}
In the rest of this paper, we replace the Riemannian norm by the uniformly equivalent box norm defined as:
\[
    \|u+v\|=\max\{\|u\|,\|v\|\},
    \qquad u\in E(x),\ v\in F(x).
\]
\end{remark}
After shrinking $\rho$ and enlarging $C_2$ if necessary, the following uniform chart estimates hold. We omit the details for the proof.

\begin{corollary}\label{cor:local-chart}
For every $0<\varepsilon\leqslant\varepsilon_0$, there exist $\rho>0$ and $C_2\geqslant1$ such that, for every $x\in\operatorname{supp}(\mu)$,
\begin{enumerate}[label=(\roman*),leftmargin=*]
    \item $\|A_x\|/m(B_x)\leqslant \lambda^N<1$;
    \item $\widetilde g_x=g_x$ on $T_xM(\rho)$;
    \item $h_x(0)=0$;
    \item $\|Dh_x\|,\|Dh_{x,1}\|,\|Dh_{x,2}\|\leqslant\varepsilon$;
    \item $\omega_{D\widetilde g_x}(r)$, $\omega_{Dh_x}(r)$, $\omega_{Dh_{x,1}}(r)$ and $\omega_{Dh_{x,2}}(r)$ are bounded by $C_2\omega(r)$ for all $r\geqslant0$.
\end{enumerate}
\end{corollary}

\subsection{Admissible manifolds and plaque families}

\begin{definition}
Let $x\in\operatorname{supp}(\mu)$ and $q\in(0,\infty]$. A $q$-admissible manifold at $x$ is a set $V\subset T_xM$ of the form
\[
    V=\{(\varphi(t),t):t\in F(x,q)\},
    \qquad F(x,q)=\{t\in F(x):\|t\|\leqslant q\},
\]
where $\varphi\colon F(x,q)\to E(x)$ is a $C^1$ map satisfying $\varphi(0)=0$ and $\|D\varphi\|\leqslant1$. The map $\varphi$ is called the representing function of $V$.
\end{definition}

The following theorem comes from the Plaque Family Theorem of Hirsch--Pugh--Shub \cite[Theorem~5.5]{HPS77}; for further regularity properties of plaque families, see also \cite{PSW97}. Similar discussions are also presented in \cite{LVY13,TY,Part2}.

\begin{theorem}\label{pft}
For every $x\in\operatorname{supp}(\mu)$, there is an $\infty$-admissible manifold $\widetilde W_x\subset T_xM$ such that
\[
    \widetilde g_x(\widetilde W_x)=\widetilde W_{gx}.
\]
If $\widetilde\varphi_x$ is the representing function of $\widetilde W_x$, then $\widetilde\varphi_x$ and $D\widetilde\varphi_x$ depend continuously on $x$ in the following sense: if $v_n=(\widetilde{\varphi}_{x_n}(t_n), t_n)\in\widetilde{W}_{x_n}$ and $x_n\to x, t_n\to t$, then $
\widetilde{\varphi}_{x_n}(t_n)\to\widetilde\varphi_x(t), D\widetilde\varphi_{x_n}(t_n)\to D\widetilde\varphi_x(t).$
\end{theorem}
\begin{remark}
The image under $\exp_x$ of the restriction of $\widetilde W_x$ to $T_xM(\rho):=\{v\in T_xM:\|v\|\leqslant \rho \}$ is usually called the local plaque tangent to $F$ at $x$.
\end{remark}

\subsection{Dini estimates under the graph transform}

The next proposition describes how the modulus of the tangent distribution of the admissible manifold changes under the graph transform.

\begin{proposition}\label{prop:graph-transform}
Let $0<\varepsilon\leqslant\varepsilon_0$, let $x\in\operatorname{supp}(\mu)$, and let $V$ be a $q$-admissible manifold at $x$ with representing function $\varphi$. Then $\widetilde g_x(V)$ contains a $(m(B_x)-\varepsilon)q$-admissible manifold at $gx$. If $\phi$ is its representing function, then
\[
\begin{aligned}
\omega_{D\phi}(r)
&\leqslant
\frac{(\|A_x\|+\varepsilon)m(B_x)}{(m(B_x)-\varepsilon)^2}
\,\omega_{D\varphi}\left(\frac{r}{m(B_x)-\varepsilon}\right) \\
&\quad +
\frac{\|A_x\|+m(B_x)}{(m(B_x)-\varepsilon)^2}
\,\omega_{Dh_x}\left(\frac{r}{m(B_x)-\varepsilon}\right).
\end{aligned}
\]
\end{proposition}

\begin{proof}
Write
\[
    \widetilde g_x(\varphi(t),t)=
    \bigl(A_x\varphi(t)+h_{x,1}(\varphi(t),t),
    B_xt+h_{x,2}(\varphi(t),t)\bigr).
\]
Set
\[
    \tau(t)=B_xt+h_{x,2}(\varphi(t),t).
\]
Since
\[
    \|D(h_{x,2}(\varphi(t),t))\|
    \leqslant \|Dh_{x,2}\|\max\{\|D\varphi\|,1\}
    \leqslant \varepsilon,
\]
the Lipschitz inverse function theorem \cite[pp.~137--138]{HP70} gives an inverse $t=t(\tau)$ on $F(gx,(m(B_x)-\varepsilon)q)$, with
\[
    \|Dt\|\leqslant \frac1{m(B_x)-\varepsilon}.
\]
Define
\[
    \phi(\tau)=A_x\varphi(t(\tau))+h_{x,1}(\varphi(t(\tau)),t(\tau)).
\]
Then
\[
\begin{aligned}
D\phi(\tau)
&=A_xD\varphi(t(\tau))Dt(\tau) \\
&\quad +Dh_{x,1}(\varphi(t(\tau)),t(\tau))
\begin{pmatrix}D\varphi(t(\tau))\\ I\end{pmatrix}Dt(\tau).
\end{aligned}
\]
The choice of $\varepsilon_0$ in \Cref{eq:epsilon0-choice} gives
\[
    \|D\phi\|
    \leqslant \frac{\|A_x\|+\varepsilon}{m(B_x)-\varepsilon}
    \leqslant 1,
\]
so the graph is admissible.

It remains to estimate the modulus. Since
\[
    Dt=\left(B_x+Dh_{x,2}(\varphi\circ t,t)
    \begin{pmatrix}D\varphi\circ t\\ I\end{pmatrix}\right)^{-1},
\]
using Lemma \ref{lem:bundle-Dini-ineq}, we have
\[
\begin{aligned}
    \omega_{Dt}(r)
    &\leqslant \frac{1}{(m(B_x)-\varepsilon)^2}
    \omega_{Dh_{x,2}(\varphi\circ t,t)
    \begin{psmallmatrix}D\varphi\circ t\\ I\end{psmallmatrix}}(r)\\
    &\leqslant \frac{1}{(m(B_x)-\varepsilon)^2}
    \left(\omega_{Dh_{x,2}(\varphi\circ t,t)}(r)
    +\varepsilon\,\omega_{D\varphi\circ t}(r)\right).
\end{aligned}
\]
Here all maps are considered in the $\tau$-coordinate. Moreover, the graph map $\tau\mapsto(\varphi(t(\tau)),t(\tau))$ is
$1/(m(B_x)-\varepsilon)$-Lipschitz in the box norm. Hence
\[
\begin{aligned}
    \omega_{D\varphi\circ t}(r)
    &\leqslant \omega_{D\varphi}\left(\frac{r}{m(B_x)-\varepsilon}\right),\\
    \omega_{Dh_{x,j}(\varphi\circ t,t)}(r)
    &\leqslant
    \omega_{Dh_x}\left(\frac{r}{m(B_x)-\varepsilon}\right),
\end{aligned}
\]
for $j=1,2$. Expanding the differences of the two products in the displayed formula for $D\phi$ and using
$\|D\varphi\|\leqslant1$, $\|Dh_{x,j}\|\leqslant\varepsilon$ and
$\|Dt\|\leqslant(m(B_x)-\varepsilon)^{-1}$, by Lemma \ref{lem:bundle-Dini-ineq} we obtain
\[
\begin{aligned}
\omega_{D\phi}(r)
&\leqslant
\frac{\|A_x\|+\varepsilon}{m(B_x)-\varepsilon}
    \omega_{D\varphi\circ t}(r)
+\frac{1}{m(B_x)-\varepsilon}
    \omega_{Dh_x(\varphi\circ t,t)}(r)\\
&\quad
+(\|A_x\|+\varepsilon)\omega_{Dt}(r).
\end{aligned}
\]
Substituting the bound for $\omega_{Dt}$, we get the asserted estimate.
\end{proof}

Set
\[
    \kappa=
    \sup_{x\in\operatorname{supp}(\mu)}
    \frac{(\|A_x\|+\varepsilon)m(B_x)}{(m(B_x)-\varepsilon)^2}.
\]

Based on the choice of $\varepsilon_0$, we have $0<\kappa<1$. Set
\[
    K=\sup_{x\in\operatorname{supp}(\mu)}
    \frac{\|A_x\|+m(B_x)}{(m(B_x)-\varepsilon)^2}<\infty.
\]
Then \Cref{prop:graph-transform} and \Cref{cor:local-chart} imply the following.

\begin{corollary}\label{cor:iterating-dini}
Under the assumptions of \Cref{prop:graph-transform},
\[
    \omega_{D\phi}(r)
    \leqslant \kappa\,\omega_{D\varphi}\left(\frac{r}{m(B_x)-\varepsilon}\right)
    +KC_2\,\omega\left(\frac{r}{m(B_x)-\varepsilon}\right).
\]
\end{corollary}

\subsection{Modulus estimates for plaques}
We first determine the sizes of the plaques.
Recall
\[
    \max\{\lambda_{k+1}(\mu,f),0\}<\theta<\lambda_k(\mu,f),
\]
and choose $\chi$ such that
\[
    N\theta<\chi<N(\lambda_k-\eta).
\]
Decreasing $\varepsilon_0$ once more if necessary, fix $\delta>0$ such that
\begin{equation}\label{eq:delta-choice}
    \chi+\delta<N(\lambda_k-\eta)
\end{equation}
and
\begin{equation}\label{eq:epsilon-delta-choice}
    \varepsilon_0\leqslant (1-e^{-\delta})
    \inf_{x\in\operatorname{supp}(\mu)}m(Dg|_{F(x)}).
\end{equation}
For $0<\varepsilon\leqslant\varepsilon_0$, define
\[
    L(x)=\min\left\{\inf_{\ell\geqslant 1}
    \prod_{i=1}^{\ell}\frac{m(Dg|_{F(g^{-i}x)})-\varepsilon}{e^\chi},1\right\}.
\]

\begin{lemma}\label{lem:size-L}
For $\mu$-almost every $x$, $L(x)>0$. Moreover,
\[
    L(gx)\leqslant \frac{m(Dg|_{F(x)})-\varepsilon}{e^\chi}L(x).
\]
\end{lemma}

\begin{proof}
By \Cref{eq:epsilon-delta-choice},
\[
    m(Dg|_{F(z)})-\varepsilon\geqslant e^{-\delta}m(Dg|_{F(z)})
\]
for all $z\in\operatorname{supp}(\mu)$. Hence
\[
\begin{aligned}
    \prod_{i=1}^{\ell}\frac{m(Dg|_{F(g^{-i}x)})-\varepsilon}{e^\chi}
    &\geqslant
    \exp\left(\sum_{i=1}^{\ell}\log m(Dg|_{F(g^{-i}x)})-
    \ell(\chi+\delta)\right).
\end{aligned}
\]
For $\mu$-almost every $x$, \Cref{eq:Birkhoff-positive-F} and \Cref{eq:delta-choice} imply that the right-hand side tends to $+\infty$ as $\ell\to\infty$. Therefore the minimum defining $L(x)$ is positive.

The second assertion follows from the defining products for $L(x)$ directly.
\end{proof}

We also need these radii to shrink subexponentially along typical backward orbits.
See \cite[Lemma 8.7]{ABC2011} for the proof.
\begin{lemma}\label{lem:L-tempered}
For every $\gamma>0$ and for $\mu$-almost every $x$, there exists $C_L(x,\gamma)>0$ such that
\[
    L(g^{-n}x)\geqslant C_L(x,\gamma)e^{-\gamma n},
    \qquad n\geqslant0.
\]
\end{lemma}

By \Cref{prop:graph-transform}, $\widetilde g_x$ maps an $L(x)\rho$-admissible manifold at $x$ into a set containing an $(m(B_x)-\varepsilon)L(x)\rho$-admissible manifold at $gx$. \Cref{lem:size-L} allows us to restrict this image to an $L(gx)\rho$-admissible manifold. Suppose the representing function of the $L(x)\rho$-admissible manifold of $x$ is $\varphi$ and the representing function of the $L(gx)\rho$-admissible manifold of $gx$ is $\phi$. The mapping from $\varphi$ to $\phi$ is called the graph transform $T$, i.e., $\phi=T(\varphi).$

For $\mu$-almost every $x$, let $\varphi_{-n}$ be the representing function of an $L(g^{-n}x)\rho$-admissible manifold at $g^{-n}x$.
Set
\[
    W_x=\left\{(\varphi_x(t),t):t\in F(x,L(x)\rho)\right\}.
\]

Define also
\[
    M(x)=\max\left\{\sup_{\ell\geqslant0}
    \prod_{i=1}^{\ell}\frac1{m(B_{g^{-i}x})-\varepsilon},1\right\}.
\]
The following lemma shows the subexponentially varying property of $M(x)$. See \cite[Lemma 8.7]{ABC2011} for the proof.
\begin{lemma}\label{lem:size-M}
For $\mu$-almost every $x$, $M(x)<\infty$. Moreover, for every $\gamma>0$ and $\mu$-almost every $x$, there exists $C_M(x,\gamma)>0$ such that
\[
    M(g^{-n}x)\leqslant C_M(x,\gamma)e^{\gamma n},
    \qquad n\geqslant0.
\]
\end{lemma}

\begin{lemma}\label{lem:Dini-Tn}
For $\mu$-almost every $x$ and every $\varphi_{-n}$ as above, denote $C_3=\frac{KC_2}{1-\kappa}$, then
\[
    \omega_{D(T^n\varphi_{-n})}(r)
    \leqslant \kappa^n\omega_{D\varphi_{-n}}(M(x)r)
    +C_3\omega(M(x)r).
\]
\end{lemma}

\begin{proof}
Iterating \Cref{cor:iterating-dini} gives
\[
\begin{aligned}
\omega_{D(T^n\varphi_{-n})}(r)
&\leqslant \kappa^n\omega_{D\varphi_{-n}}
\left(r\prod_{i=1}^n\frac{1}{m(B_{g^{-i}x})-\varepsilon}\right) \\
&\quad +\sum_{j=0}^{n-1}\kappa^j KC_2
\omega\left(r\prod_{i=1}^{j+1}\frac{1}{m(B_{g^{-i}x})-\varepsilon}\right).
\end{aligned}
\]
The definition of $M(x)$ and the monotonicity of $\omega$ give the claimed bound.
\end{proof}

\begin{corollary}\label{cor:plaque-Dini}
For $\mu$-almost every $x$,
\[
    \omega_{D\varphi_x}(r)
    \leqslant C_3\omega(M(x)r).
\]
In particular, $W_x$ is a $C^{1+\omega}$ plaque.
\end{corollary}

\begin{proof}
Take $\varphi_{-n}=\varphi_{g^{-n}x}$. The invariance of the plaque family gives $T^n\varphi_{-n}=\varphi_x$. Since $\|D\varphi_{-n}\|\leqslant1$,
\[
    \omega_{D\varphi_x}(r)
    \leqslant 2\kappa^n+C_3\omega(M(x)r).
\]
Letting $n\to\infty$ proves the estimate. By \Cref{lem:basic-modulus},
$\omega(M(x)r)\leqslant M(x)\omega(r)$. Thus $D\varphi_x$ is $\omega$-continuous, and $W_x$ is a $C^{1+\omega}$ plaque.
\end{proof}
\begin{remark}\label{rem:nonuniform-plaque-Dini}
The plaque family is defined at every point of
$\operatorname{supp}(\mu)$, but the regularity conclusion of
\Cref{cor:plaque-Dini} is only an almost-everywhere, nonuniform one.
Indeed, the quantity $M(x)$ is known to be finite only for
$\mu$-almost every $x$, and on this full-measure set the estimate gives
\[
    [D\varphi_x]_\omega\leqslant C_3M(x),
\]
where $M(x)$ could be unbounded. The counterexample for the uniform estimates of the modulus of continuity can be found even when $\omega$ is H\"older. 
\end{remark}

\begin{remark}\label{rem:Dini-neighborhood}
The Dini-continuity estimates for the plaques are essential to the
entropy-formula characterization of SRB measures. Indeed, the plaques need not
be contained in $\operatorname{supp}(\mu)$; consequently, a modulus of continuity
for the dominated splitting restricted to $\operatorname{supp}(\mu)$ is insufficient for the
distortion argument and hence for the proof of the main theorem.
\end{remark}

\subsection{Bounded distortion on plaques}

Let $\widetilde g\colon TM\to TM$ be the bundle map whose fiber map over $x$ is $\widetilde g_x$. For $v\in T_xM$ and $n\geqslant0$, write
\[
    \widetilde g^{-n}(v)=
    \widetilde g_{g^{-n}x}^{-1}\circ\cdots\circ
    \widetilde g_{g^{-2}x}^{-1}\circ\widetilde g_{g^{-1}x}^{-1}(v).
\]

\begin{theorem}\label{thm:plaque-distortion}
For $\mu$-almost every $x$ and every $v,w\in W_x$, the following assertions hold.
\begin{enumerate}[label=(\arabic*),leftmargin=*]
    \item For every $\ell\geqslant0$,
    \[
        \|\widetilde g^{-\ell}v-\widetilde g^{-\ell}w\|
        \leqslant 2e^{-\ell\chi}\rho.
    \]
    \item Let
    \[
        \widetilde J^k(u)=\left|\operatorname{Jac}
        \left(D\widetilde g|_{T_uW_{\pi(u)}}\right)\right|.
    \]
    Then
    \[
        \prod_{\ell=1}^{\infty}
        \frac{\widetilde J^k(\widetilde g^{-\ell}v)}
             {\widetilde J^k(\widetilde g^{-\ell}w)}
    \]
    converges. The limit is bounded away from $0$ and $\infty$ by constants depending on $x$, but not on $v,w\in W_x$.
\end{enumerate}
\end{theorem}

\begin{proof}
The first assertion follows by reading the expansion of the $F$-coordinate backwards. The second assertion then follows by summing the variations of the logarithmic Jacobian along the exponentially contracting backward orbit.

Write
\[
    \widetilde g^{-i}v=(\varphi_{-i}(v_{-i}),v_{-i}),
    \qquad
    \widetilde g^{-i}w=(\varphi_{-i}(w_{-i}),w_{-i}),
\]
where $\varphi_{-i}=\varphi_{g^{-i}x}$. Since $\|D\varphi_{-i}\|\leqslant1$ and we use the box norm,
\[
    \|\widetilde g^{-i}v-\widetilde g^{-i}w\|=\|v_{-i}-w_{-i}\|.
\]
The $F$-coordinate of the graph transform gives
\[
    \|v_{-i+1}-w_{-i+1}\|
    \geqslant (m(B_{g^{-i}x})-\varepsilon)\|v_{-i}-w_{-i}\|.
\]
Therefore
\[
\begin{aligned}
    \|\widetilde g^{-\ell}v-\widetilde g^{-\ell}w\|
    &\leqslant
    \prod_{i=1}^{\ell}\frac1{m(B_{g^{-i}x})-\varepsilon}\|v-w\| \\
    &\leqslant L(x)^{-1}e^{-\ell\chi}\,2L(x)\rho
    =2e^{-\ell\chi}\rho.
\end{aligned}
\]
This proves (1).

For (2), it suffices to show that the logarithms form an absolutely convergent series. The restriction of $D\widetilde g$ to the tangent spaces of the plaques has conorm bounded below by $\inf_x(m(B_x)-\varepsilon)>0$, so the logarithm of the $k$-Jacobian is locally Lipschitz as a function of the linear map and of the tangent $k$-plane. Since $D\widetilde g_x$ is $\omega$-continuous and the tangent space of $W_x$ is the graph of $D\varphi_x$, there exists $C_4>0$ such that, for points on the same plaque,
\[
    |\log\widetilde J^k(z)-\log\widetilde J^k(z')|
    \leqslant C_4\left(\omega_{D\widetilde g}(\|z-z'\|)+
    \|D\varphi(z)-D\varphi(z')\|\right).
\]
Using (1), \Cref{cor:local-chart}, \Cref{cor:plaque-Dini}, and \Cref{lem:size-M} with $\gamma=\chi/2$, we obtain, for $\mu$-almost every $x$,
\[
\begin{aligned}
&|\log\widetilde J^k(\widetilde g^{-i}v)-
\log\widetilde J^k(\widetilde g^{-i}w)| \\
&\qquad\leqslant
C_4\left(C_2\omega(2e^{-i\chi}\rho)+
C_3\omega(2C_M(x,\chi/2)\rho e^{-i\chi/2})\right).
\end{aligned}
\]
The right-hand side is summable in $i$ by \Cref{prop:dini-summability}. Hence, the logarithm of the distortion product converges and is bounded by a constant depending only on $x$. This proves (2).
\end{proof}

\begin{corollary}\label{key}
    For $\mu$-almost every $x$ and every $y\in W^k(x)$,
    $$
    \prod_{i=1}^{\infty}
    \frac{\bar J^k(g^{-i}y)}{\bar J^k(g^{-i}x)},
    \qquad
    \bar J^k(y)\coloneqq
    \left|
        \operatorname{Jac}
        \left(Dg|_{T_yW^k(x)}\right)
    \right|
    $$
    is bounded away from $0$ and $\infty$ on any compact subset of
    $W^k(x)$.
\end{corollary}

\begin{proof}
    Note that, for $\mu$-almost every $x$, $W_x$ is an
    $L(x)\rho$-admissible manifold. Since $L(x)\leqslant1$ and
    $\widetilde g_x(v)=g_x(v)$ for $\|v\|\leqslant\rho$, according to
    item (1) of \Cref{thm:plaque-distortion},
    $\exp_x(W_x)\subset W^k(x)$ and forms a compact neighborhood of $x$
    in $W^k(x)$. Since $L(x)$ varies subexponentially and, for any
    $y\in W^k(x)$, $d(g^{-n}y,g^{-n}x)$ tends to $0$ exponentially fast,
    it is easy to see that there exists an integer $m\geqslant1$ such that
    \[
        g^{-m}y\in
        \exp_{g^{-m}x}(W_{g^{-m}x}).
    \]
    Moreover, for any compact neighborhood $V$ of $x$ in $W^k(x)$,
    there exists an integer $m\geqslant1$ such that
    \[
        g^{-m}y\in
        \exp_{g^{-m}x}(W_{g^{-m}x})
    \]
    for any $y\in V$.

    Since
    $$
    \prod_{i=1}^{\infty}
    \frac{\bar J^k(g^{-i}y)}{\bar J^k(g^{-i}x)}
    =
    \prod_{i=1}^{m-1}
    \frac{\bar J^k(g^{-i}y)}{\bar J^k(g^{-i}x)}
    \cdot
    \prod_{i=m}^{\infty}
    \frac{\bar J^k(g^{-i}y)}{\bar J^k(g^{-i}x)},
    $$
    we only have to show that
    $$
    \prod_{i=m}^{\infty}
    \frac{\bar J^k(g^{-i}y)}{\bar J^k(g^{-i}x)}
    $$
    is bounded away from $0$ and $\infty$ on $V$.

    For simplicity, assume that $m=1$, $y=\exp_x(v)$, and
    $v\in T_xM(L(x)\rho)$. For any $\ell\geqslant j\geqslant m=1$,
    $$
    \prod_{i=j}^{\ell}\bar J^k(g^{-i}y)
    =
    \prod_{i=j}^{\ell}
    \left|
        \operatorname{Jac}
        \left(
            Dg|_{T_{g^{-i}y}W^k(g^{-i}x)}
        \right)
    \right|
    =
    \left|
        \operatorname{Jac}
        \left(
            Dg^{\ell-j+1}
            |_{T_{g^{-\ell}y}W^k(g^{-\ell}x)}
        \right)
    \right|.
    $$
    Naturally, here the Jacobian is with respect to the inner products
    on $T_{g^{-\ell}y}M$ and $T_{g^{-j+1}y}M$. But the Jacobian in
    $$
    \prod_{i=j}^{\ell}
    \widetilde J^k(\widetilde g^{-i}v)
    =
    \prod_{i=j}^{\ell}
    \left|
        \operatorname{Jac}
        \left(
            D\widetilde g
            |_{T_{\widetilde g^{-i}v}W_{g^{-i}x}}
        \right)
    \right|
    =
    \left|
        \operatorname{Jac}
        \left(
            D\widetilde g^{\ell-j+1}
            |_{T_{\widetilde g^{-\ell}v}W_{g^{-\ell}x}}
        \right)
    \right|
    $$
    is, of course, with respect to the inner products on
    $T_{g^{-\ell}x}M$ and $T_{g^{-j+1}x}M$.

    Since the inner product on $T_zM$ varies $C^\infty$-smoothly with
    $z$, there exists a constant $B>0$ such that, for every
    $v\in T_xM(L(x)\rho)$ and $\ell\geqslant j\geqslant m$,
    $$
    \left|
        \frac{
            \prod_{i=j}^{\ell}
            \widetilde J^k(\widetilde g^{-i}v)
        }{
            \prod_{i=j}^{\ell}
            \bar J^k(g^{-i}y)
        }
        -1
    \right|
    \leqslant
    B\left(
        d(g^{-\ell}y,g^{-\ell}x)
        +
        d(g^{-j+1}y,g^{-j+1}x)
    \right).
    $$
    Now, the conclusion of the corollary follows from
    \Cref{thm:plaque-distortion}.
\end{proof}

\section{Proof of the Main Theorem}\label{sec:proof-main}

This section applies the geometric distortion estimate from \Cref{sec:geometric} to obtain the absolute continuity of conditional measures. The argument follows the ideas in \cite[Section 6]{LY2}. We present it here to make the Dini-regularity argument self-contained.

We prove \Cref{main-thm-intro}. The implication from absolute continuity to the partial entropy formula is proved in \cite{WWZ}. We prove the converse. Let $\xi$ be an increasing measurable partition subordinate to $W^k$ as in \cite{LS82,TY,Part2,WWZ}. We choose it so that its atoms contain local neighborhoods inside the corresponding Pesin unstable manifolds. Thus, $(f^{-1}\xi)(x)\subset \xi(x)$ for $\mu$-almost every $x$, and the restriction of $f^{-1}\xi$ to each atom $\xi(x)$ is countable.

If $z$ belongs to a Pesin unstable manifold $W^k(x)$, we write
\[
    J^k(z)=\left|\operatorname{Jac}\left(Df|_{T_zW^k(x)}\right)\right|,
\]
suppressing the leaf from the notation. This convention is harmless because all points under consideration remain on the corresponding iterates of the same unstable leaf. For $\mu$-almost every $x$ and every $y\in\xi(x)$, define
\begin{equation}\label{eq:Delta-def}
    \Delta(x,y)=\prod_{\ell=1}^{\infty}
    \frac{J^k(f^{-\ell}x)}{J^k(f^{-\ell}y)}.
\end{equation}
The convergence and boundedness of this product on atoms of $\xi$ follow from \Cref{key}. Indeed, since $z$ lies on a Pesin unstable leaf $W$ and $g=f^N$, we have
\[
    \bar J^k(z)=\prod_{r=0}^{N-1}J^k(f^r z).
\]
Thus, the product in \Cref{eq:Delta-def} is obtained by regrouping the $g$-distortion product into blocks of length $N$. The bounded distortion estimate for $g$ therefore gives the required convergence and boundedness for $f$.

Let $m_x$ be the Riemannian volume on $W^k(x)$. Since $\Delta(x,\cdot)$ is positive and bounded on $\xi(x)$, the quantity
\[
    R(x)=\int_{\xi(x)}\Delta(x,y)\,dm_x(y)
\]
is positive and finite for $\mu$-almost every $x$. Define probability measures on the atoms of $\xi$ by
\[
    d\nu_x(y)=\frac{\Delta(x,y)}{R(x)}\,dm_x(y).
\]
This is the natural candidate for the disintegration: the density is chosen so that its transformation under $f$ has exactly the Jacobian prescribed by \Cref{eq:Delta-def}. Together with the quotient measure induced by $\mu$ on the space of atoms of $\xi$, these conditional measures define a probability measure $\nu$. By construction, $\nu$ and $\mu$ coincide on the $\sigma$-algebra $\mathscr B_\xi$ of $\xi$-saturated measurable sets.

We use the following elementary coboundary lemma; see \cite[Proposition 2.2]{LS82} for the proof.

\begin{lemma}\label{lem:coboundary-integrability}
Let $(X,\mathcal B,\mu,T)$ be an invertible probability-preserving system. Let $R\colon X\to(0,\infty)$ be a finite measurable function, and set
\[
    u=\log\frac{R\circ T}{R}.
\]
If $u^+:=\max\{u,0\}\in L^1(\mu)$, then $u\in L^1(\mu)$ and $\int u\,d\mu=0$.
\end{lemma}

\begin{lemma}\label{lem:log-equality}
We have
\[
    \int -\log \nu_x((f^{-1}\xi)(x))\,d\mu(x)
    =\int \log J^k(x)\,d\mu(x).
\]
\end{lemma}

\begin{proof}
Set
\[
    q(x)=\nu_x((f^{-1}\xi)(x)).
\]
Using the identity
\[
    \Delta(fx,fy)=\frac{J^k(x)}{J^k(y)}\Delta(x,y)
\]
and changing variables by $f$ along the leaf $W^k(x)$, we obtain
\[
\begin{aligned}
    q(x)
    &=\frac{1}{R(x)}\int_{(f^{-1}\xi)(x)}\Delta(x,y)\,dm_x(y) \\
    &=\frac{R(fx)}{R(x)}\cdot\frac{1}{J^k(x)}.
\end{aligned}
\]
Since $0<q(x)\leqslant1$, we have
\[
    \left(\log\frac{R(fx)}{R(x)}\right)^+
    \leqslant \log^+J^k(x).
\]
The function $\log^+J^k$ is integrable because $M$ is compact and $f$ is $C^1$. \Cref{lem:coboundary-integrability} therefore gives
\[
    \int \log\frac{R(fx)}{R(x)}\,d\mu(x)=0.
\]
Taking logarithms in the formula for $q(x)$ and integrating with $\mu$ proves the claim.
\end{proof}

\begin{lemma}\label{lem:nu-mu-on-fxi}
If
\[
    H_\mu(\xi\mid f\xi)=\int\log J^k\,d\mu,
\]
then $\nu$ and $\mu$ coincide on $\mathscr B_{f^{-1}\xi}$.
\end{lemma}

\begin{proof}
For $\mu$-almost every atom $\xi(x)$, the partition $(f^{-1}\xi)|_{\xi(x)}$ is countable. Denote its atoms by $\{C_j(x)\}_j$, and write
\[
    p_j(x)=\mu_x(C_j(x)),
    \qquad
    q_j(x)=\nu_x(C_j(x)).
\]
Applying the concavity of the logarithm to the set of indices for which $p_j(x)>0$ gives, for almost every $x$,
\[
    \sum_{p_j(x)>0} p_j(x)\log\frac{q_j(x)}{p_j(x)}
    \leqslant \log\sum_{p_j(x)>0} q_j(x)\leqslant0.
\]
Equality holds if and only if $q_j(x)=0$ whenever $p_j(x)=0$ and, on the remaining indices, $q_j(x)=p_j(x)$ for all $j$.

Integrating the last inequality over the quotient space of $\xi$, we obtain
\[
    \int \log\frac{\nu_y((f^{-1}\xi)(y))}
                    {\mu_y((f^{-1}\xi)(y))}\,d\mu(y)
    \leqslant 0.
\]
On the other hand, \Cref{lem:log-equality} and the assumed entropy identity give
\[
\begin{aligned}
    \int -\log \nu_x((f^{-1}\xi)(x))\,d\mu(x)
    &=\int\log J^k\,d\mu \\
    &=H_\mu(\xi\mid f\xi) \\
    &=\int -\log \mu_x((f^{-1}\xi)(x))\,d\mu(x).
\end{aligned}
\]
Therefore, the integrated logarithmic inequality is an equality. Hence, for almost every $\xi$-atom, the probability vectors $(p_j(x))_j$ and $(q_j(x))_j$ coincide. This means precisely that $\nu=\mu$ on $\mathscr B_{f^{-1}\xi}$.
\end{proof}
 
Using inductively \Cref{lem:nu-mu-on-fxi}, we obtain the following lemma.

\begin{lemma}\label{lem:iterated-equality}
Assume that
\[
    H_\mu(\xi\mid f\xi)=\int\log J^k\,d\mu.
\]
Then $\nu$ and $\mu$ coincide on $\mathscr B_{f^{-n}\xi}$ for every $n\geqslant1$.
\end{lemma}

\begin{proof}[Proof of \Cref{main-thm-intro}]
Assume
\[
    h^k_\mu(f)=\sum_{i=1}^k\lambda_i(\mu,f).
\]
Since $T_xW^k(x)=F(x)$ for $\mu$-almost every $x$, Oseledets' theorem gives
\[
    \int\log J^k\,d\mu
    =\int\log\left|\det(Df|_{F})\right|\,d\mu
    =\sum_{i=1}^k\lambda_i(\mu,f).
\]
Thus the assumed partial entropy formula is equivalent to
\[
    h^k_\mu(f)=\int\log J^k\,d\mu.
\]
By \Cref{lem:iterated-equality}, $\mu$ and $\nu$ coincide on $\mathscr B_{f^{-n}\xi}$ for every $n\geqslant1$. Since $\xi$ is subordinate to the unstable manifolds and its backward iterates separate points, we have $\mu=\nu$. In particular, the conditional measures of $\mu$ on the atoms of $\xi$ are
\[
    d\mu_x(y)=\frac{\Delta(x,y)}{R(x)}\,dm_x(y),
\]
and are absolutely continuous with respect to the leafwise Riemannian volume on $W^k(x)$. This proves the converse implication. The other implication is proved in \cite{WWZ}, and the theorem follows.
\end{proof}

\section{Application: mostly expanding attractors}\label{sec:application}

This section proves the application to the mostly expanding attractors stated in the introduction, i.e., \Cref{coro2}. The proof combines smooth approximation, upper semicontinuity of partial entropy, and \Cref{main-thm-intro}; the main point is to pass the center-unstable entropy formula from smooth approximants to the Dini-regular map.

We first recall the upper semicontinuity statement for partial entropy in the form used below; see \cite[Section 1.2.1]{TY2}.

\begin{proposition}\label{prop:partial-entropy-semicontinuity}
Let $f_n\in\mathrm{Diff}^{1+\alpha}(M)$ converge to $f\in\mathrm{Diff}^1(M)$ in the $C^1$ topology, and let $\mu_n$ be $f_n$-invariant ergodic probability measures converging in the weak-$*$ topology to an $f$-invariant probability measure $\mu$. Let $k\in\mathbb N$. Assume that, for each $n$, there is an $L$-dominated splitting
\[
    T_{\operatorname{supp}(\mu_n)}M=E_n\oplus F_n,
    \qquad \dim F_n=k,
\]
for $f_n$, and that there is an $L$-dominated splitting
\[
    T_{\operatorname{supp}(\mu)}M=E\oplus F,
    \qquad \dim F=k,
\]
for $f$. Moreover, assume that
\[
    \operatorname*{ess\,inf}_{x\sim\mu}\lambda_k(x,f)>0.
\]
Then $h^k_{\mu_n}(f_n)$ and $h^k_\mu(f)$ are well-defined for all sufficiently large $n$, and
\[
    \limsup_{n\to\infty}h^k_{\mu_n}(f_n)\leqslant h^k_\mu(f).
\]
\end{proposition}

\begin{proof}[Proof of \Cref{coro2}]
Let $U$ be a trapping region defining $\Lambda$ so that
\(
    \Lambda=\bigcap_{j\geqslant0}f^j(\overline U).   
\)
By the $C^1$-openness of partial hyperbolicity, every $C^1$-nearby diffeomorphism $g$ has an attractor
\[
    \Lambda_g=\bigcap_{j\geqslant0}g^j(\overline U)
\]
with the corresponding continuation of the partially hyperbolic splitting. The Gibbs-$u$ condition is also stable under small $C^1$ perturbations in the following sense. If $g_j\to f$ in the $C^1$ topology and $\nu_j$ is a Gibbs-$u$ state of $g_j|_{\Lambda_{g_j}}$, then every weak-$*$ accumulation point of the sequence $\nu_j$ satisfies the $u$-entropy formula by upper semicontinuity of partial entropy along the uniformly expanding foliation \cite{Yang21,TY2}; applying \Cref{main-thm-intro} with $F=E^u$ shows that the limit is a Gibbs-$u$ state of $f|_\Lambda$. Since $E^c_g$ varies continuously with $g$, the inequality
\[
    \int \log m(Dg^{N_0}|_{E^c_g})\,d\nu \geqslant c_0/2
\]
holds for all such $g$ close to $f$ and all Gibbs-$u$ states $\nu$ of $g|_{\Lambda_g}$. Hence we may choose $C^\infty$ diffeomorphisms $f_n$ converging to $f$ in the $C^1$ topology such that each $\Lambda_n=\Lambda_{f_n}$ is a mostly expanding partially hyperbolic attractor with splitting
\[
    T_{\Lambda_n}M=E^s_n\oplus E^c_n\oplus E^u_n.
\]
The dimensions agree with the corresponding dimensions for $f$, and the splittings converge uniformly to the splitting of $f$. 

For the smooth approximants, we use the $C^{1+\alpha}$ theory for mostly expanding systems in the Gibbs-$u$ formulation \cite{AnderssonVasquez2018}, which is based on the nonuniform expansion mechanism initiated by Alves--Bonatti--Viana \cite{ABV}. Thus each $f_n|_{\Lambda_n}$ admits an ergodic SRB measure $\mu_n$ whose center Lyapunov exponents are positive. Consequently, the Pesin unstable manifolds of $\mu_n$ are tangent to $E^{cu}_n=E^c_n\oplus E^u_n$. The entropy formula of Ledrappier--Young \cite{LY2,LY85} gives
\begin{equation}\label{eq:mun-cu-formula}
    h^{cu}_{\mu_n}(f_n)
    =\int \log\left|\det(Df_n|_{E^{cu}_n})\right|\,d\mu_n
    =\sum_{i=1}^{\dim E^{cu}}\lambda_i(\mu_n,f_n).
\end{equation}

Passing to a subsequence, assume that $\mu_n\to\mu$ in the weak-$*$ topology. The limit measure $\mu$ is $f$-invariant and supported on $\Lambda$. By previous discussions, the limit measure $\mu$ is a Gibbs-$u$ state for $f$. 

The uniform consequence of the mostly expanding assumption recalled in the introduction gives $N_0\geqslant1$ and $c_0>0$ such that every Gibbs-$u$ state $\nu$ satisfies
\[
    \int \log m(Df^{N_0}|_{E^c})\,d\nu\geqslant c_0.
\]
The ergodic components of a Gibbs-$u$ state are again Gibbs-$u$ states. Applying the last inequality to the ergodic components of $\mu$ shows that their smallest center Lyapunov exponent is at least $c_0/N_0$. Since $E^u$ is uniformly expanding, the smallest exponent along $E^{cu}$ is bounded away from zero for $\mu$-almost every point. Thus \Cref{prop:partial-entropy-semicontinuity} applies with $F=E^{cu}$.
 
The function
\[
    (g,\nu)\longmapsto \int \log\left|\det(Dg|_{E^{cu}_g})\right|\,d\nu
\]
is continuous with respect to $C^1$ convergence of the diffeomorphisms and weak-$*$ convergence of the measures, because the dominated bundles vary continuously. Taking limits in \Cref{eq:mun-cu-formula} and using \Cref{prop:partial-entropy-semicontinuity}, we obtain
\[
    h^{cu}_\mu(f)
    \geqslant \int \log\left|\det(Df|_{E^{cu}})\right|\,d\mu.
\]
The reverse inequality follows from the partial Ruelle inequality \cite{WWZ}. Hence equality holds:
\[
    h^{cu}_\mu(f)
    = \int \log\left|\det(Df|_{E^{cu}})\right|\,d\mu.
\]
Since partial entropy and the Lyapunov integral are affine with respect to the ergodic decomposition, there is an ergodic component, still denoted by $\mu$, satisfying the same equality. \Cref{main-thm-intro} applied to the dominated splitting
\[
    E^s\oplus (E^c\oplus E^u)
\]
then gives absolute continuity on the partial Pesin unstable manifolds tangent to $E^{cu}$. For this ergodic component, the center exponents are positive, while $E^s$ is uniformly contracting; hence these manifolds are the full Pesin unstable manifolds. Thus $\mu$ is an SRB measure supported on $\Lambda$.
\end{proof}

\section{\texorpdfstring{$C^1$}{C1} counterexamples}\label{sec:c1-example}

This final section recalls low-regularity counterexamples for comparison. They are independent of the proof of the main theorem and illustrate why one should not expect a direct $C^1$ analogue without additional distortion control.

The following construction is adapted from \cite{Boukhecham2024} and gives a class of $C^1$ counterexamples.

Let $\mathcal{T}=D^2\times S^1$ be the solid torus, where $$
D^2=\{(x,y)\in\mathbb{R}^2\colon x^2+y^2\leqslant 1\}, \quad S^1=[0, 2\pi] \mod 2\pi.$$ 
Let $f\colon  \mathcal{T}\to \mathcal{T}$
be the map defined on the solid torus  $\mathcal{T}$ by
\[
    f(x,y,\theta)=
    \left(\frac14x+\frac12\cos\theta,
          \frac14y+\frac12\sin\theta,
          T(\theta)\right),
\]
where $T\colon S^1\to S^1$ is a $C^1$ small perturbation of the doubling map on the circle so that $f$ is a $C^1$ diffeomorphism onto its image. As a $C^1$ small perturbation of the standard solenoid map, the map is hyperbolic on its invariant set $\Lambda=\bigcap_{n\geqslant 0}f^n\mathcal{T}$. Moreover, from \cite{Quas99} we may assume that $T$ has no absolutely continuous invariant probability measure. 

\begin{lemma}\label{lem:no-srb}
The map $f$ defined above has no SRB measure.
\end{lemma}

\begin{proof}
 Without loss of generality, we assume that $T(0)=0$. It is easily seen that for every $p=(x,y,\theta)\in\Lambda$, there exists a $C^1$ map $\phi_p\colon\mathbb{R}\to D^2$ such that $\phi_p(\theta)=(x,y)$ and the graph of $\phi_p$ is the unstable manifold of $p$. Since $T(0)=0$, let $D=D^2\times\{0\}$, 
$$
\bigcup_{p\in \Lambda\cap D}\{(\phi_p(t), t)\colon t\in[0,2\pi)\}
$$
is an increasing measurable partition subordinate to the unstable manifolds. We will denote this partition by $\xi$.

Suppose, on the contrary, that there is an invariant probability measure $\mu$ whose conditional measures on unstable manifolds are absolutely continuous with respect to the leafwise Riemannian volume. 

For $p\in D\cap \Lambda$, let $\mu_p$ be the conditional measure on $\xi(p)$. Let
\[
    \pi\colon  D^2\times S^1\to S^1, (x, y, \theta)\mapsto \theta
\]
be the projection, and denote by $\pi_p$ its restriction to $\xi(p)$. Set $\nu=\pi_*\mu$. According to the characterization of the conditional measures $\{\mu_p\}$, there exists a probability measure $\tau$ on $D\cap\Lambda$ such that, for every Borel set $E\subset S^1$,
\[
    \nu(E)=\mu(\pi^{-1}E)=\mu(\bigcup_{p\in D\cap \Lambda}\pi_p^{-1}E)=\int_{D\cap \Lambda} \mu_p(\pi_p^{-1}(E))\,d\tau(p).
\]
 Note that $\pi_p\colon \xi(p)\to S^1$ is a $C^1$ local diffeomorphism. Hence, if $E\subset S^1$ has zero Lebesgue measure, then $\pi_p^{-1}(E)$ has zero Riemannian volume on $\xi(p)$. By the absolute continuity of $\mu_p$,
\[
    \mu_p(\pi_p^{-1}(E))=0
\]
for $\tau$-almost every $p$, and so $\nu(E)=0$. Thus, $\nu$ is an absolutely continuous invariant probability measure for $T$, contradicting the choice of $T$.
\end{proof}

\section*{Acknowledgements}
We would like to thank Professor Jiagang Yang for helpful suggestions and comments on earlier drafts
of this paper. 

The authors are partially supported by National Key R\&D Program of China 2022YFA1005801.

\smallskip
\noindent {\bf Data availability}
No numerical or categorical data were used in this article.

\bibliographystyle{abbrv}
{\footnotesize\bibliography{library}}

\end{document}